\documentclass[10pt]{amsart}
\usepackage{setspace,xypic,somedefs,tracefnt,latexsym,rawfonts,graphicx,amsfonts,amssymb,hyphenat,latexsym,enumerate,amscd}
\makeatletter
\@namedef{subjclassname@2020}{%
  \textup{2020} Mathematics Subject Classification}
\makeatother
\def\cal{\mathcal}
\def\Bbb{\mathbb}

\def\r{\rangle}
\def\l{\langle}

\def\wt{\widetilde}

\def\bar{\overline}
\newtheorem{thm}{Theorem}[section]

\newtheorem{lemma}[thm]{Lemma}

\newtheorem{defn}[thm]{Definition}
\newtheorem{rem}[thm]{Remark}

\numberwithin{equation}{section}
\makeatletter
\newcommand{\colim@}[2]{%
  \vtop{\m@th\ialign{##\cr
    \hfil$#1\operator@font colim$\hfil\cr
    \noalign{\nointerlineskip\kern1.5\ex@}#2\cr
    \noalign{\nointerlineskip\kern-\ex@}\cr}}%
}
\newcommand{\colim}{%
  \mathop{\mathpalette\colim@{\rightarrowfill@\textstyle}}\nmlimits@
}
\makeatother
\begin{document}
\date{\today}
\title[A four-term exact sequence of surface orbifold pure braid
groups]{A four-term exact sequence of surface orbifold pure braid groups}
\author[S.K. Roushon]{S.K. Roushon}
\address{School of Mathematics\\
Tata Institute\\
Homi Bhabha Road\\
Mumbai 400005, India}
\email{roushon@math.tifr.res.in} 
\urladdr{https://mathweb.tifr.res.in/\~\!\!\! roushon/}
\begin{abstract} We prove a four-term exact sequence of
  surface orbifold pure braid groups for all genus $\geq 1$, $2$-dimensional
  orientable orbifolds with cone points. This also corrects and
  generalizes our earlier result (genus zero case) in \cite{Rou20} and \cite{Rou21}.
\end{abstract}
 
\keywords{Orbifold, orbifold fundamental group, configuration space.}

\subjclass[2020]{Primary: 22A22, 14N20, 20F36; Secondary: 55R80.}
\maketitle

\section{Introduction}
Let $M$ be a connected smooth manifold of dimension $\geq 2$. Let 
$PB_n(M)$ be the configuration space of ordered 
$n$-tuples of pairwise distinct points of $M$. Then, the Fadell-Neuwirth
fibration theorem (\cite{FN62}) says that, for $n\geq 2$, the projection map $M^n\to M^{n-1}$
to the first $n-1$ coordinates defines a fibration
$f(M):PB_n(M)\to PB_{n-1}(M)$, with fiber homeomorphic to $\wt M:=M-\{(n-1)\
\text{points}\}$. Hence $f(M)$ induces the following long exact sequence of homotopy
groups.
\begin{align}\label{1.1}
 \xymatrix@C-1.2pc{\cdots\ar[r]&\pi_2(PB_{n-1}(M))\ar[r]&\pi_1(\wt M)\ar[r]&\pi_1(PB_n(M))\ar[r]&\pi_1(PB_{n-1}(M))\ar[r]&1.}\end{align}

It is an important subject to study the 
homotopy groups, especially the fundamental groups of the 
configuration spaces of a manifold. Since in dimension $\geq 3$, the space 
$PB_n(M)$ and the product manifold $M^n$ have isomorphic fundamental groups, 
the dimension $2$ case is of much interest. For $M={\Bbb C}$, $\pi_1(PB_n(M))$ is known as the 
pure braid group on $n$ strings. The braid groups were introduced in \cite{Ar47}, and
appear in a wide range of areas in 
both Mathematics and Physics. 

In \cite{Rou20} we studied the possibility of extending the Fadell-Neuwirth fibration
theorem for orbifolds, to understand a certain class of Artin groups.
Orbifolds are also of fundamental
importance in algebraic and differential geometry, topology and
string theory. However, 
to define a fibration between orbifolds, we had to consider the category of 
Lie groupoids. Since an orbifold can be realized as a Lie groupoid (\cite{Moe02}),
and there are enough tools in this category to define a fibration.
There, we defined two notions ($a$ and $b$-types)
of a fibration ([\cite{Rou20}, Definition 2.4]) and the corresponding ($a$ and $b$-types)
configuration Lie groupoids of a   
Lie groupoid to enable us to state a Fadell-Neuwirth type theorem.
For an orbifold $M$, the $b$-type configuration Lie groupoid is the correct
model to induce the orbifold structure on $PB_n(M)$. We proved that the
Fadell-Neuwirth fibration theorem extends in this generality, under some strong hypothesis
($c$-groupoid). We also showed that this is the best possible extension. For 
this, we deduced that
the map $f(M)$ is not a $a$(or $b$)-type fibration for the $a$(or
$b$)-type configuration 
Lie groupoids of Lie groupoids, corresponding to global quotient
compact orbifolds of dimension $\geq 2$, with
non-empty singular set (See [\cite{Rou23-1}, Proposition
3.1]). Recently, in \cite{JF23} Flechsig corrected the  
short exact sequence of fundamental groups of the $b$-type configuration Lie
groupoids, we deduced in [\cite{Rou20}, Theorem 2.14, Remark 2.15] corresponding
to all genus zero $2$-dimensional orbifolds    
with cone points and at least one puncture. He showed that it is in fact a four-term
exact sequence. This also implies that $f(M)$ is not even a
`quasifibration' of orbifolds. See Remark \ref{les}. 

In this paper we correct the analogous short exact sequence, for
all genus $\geq 1$, $2$-dimensional orbifolds with cone points, 
we gave in \cite{Rou21}, and establish a similar four-term exact sequence.
Using \cite{JF23}, we further show that it is not a short exact
sequence if the orbifold has a cone point (Theorem \ref{ig}).

The
main ingredients of the proof of Theorem \ref{ig} is the presentation of the
surface pure braid groups of an orientable surface of genus $\geq 1$, from
\cite{Bel04}. 

\section{A Four-term exact sequence}\label{MR}
For a connected $2$-dimensional orbifold $M$, $PB_n(M)$ is again
an orbifold.

\begin{defn}{\rm 
The orbifold fundamental group
of  $PB_n(M)$ is called the {\it surface orbifold pure braid} group
of $M$ on $n$ strings.}\end{defn}

Let ${\cal C}_0$ be the class of all connected, genus zero $2$-dimensional
orbifolds with cone points and at least one puncture, and
let ${\cal C}_1$ be the class of all connected, genus $\geq 1$, $2$-dimensional orientable orbifolds  
with cone points. If $M\in{\cal C}_0\cup {\cal C}_1$ has nonempty boundary, then
we replace each boundary component of $M$ by a puncture.

We prove the following theorem.

  \begin{thm}\label{ig} There is a four-term exact sequence of surface orbifold pure braid
  groups of $M\in {\cal C}_1$, as follows.
  
\centerline{
 \xymatrix@C-.6pc{1\ar[r]&K(M, n-1)\ar[r]&\pi_1^{orb}(\wt M)\ar[r]&\pi_1^{orb}(PB_n(M))\ar[r]^{f(M)_*}&\pi_1^{orb}(PB_{n-1}(M))\ar[r]&1.}}
\noindent
Here $\wt M=M-\{(n-1)\ \text{smooth points}\}$. Furthermore, if $M$
has a cone point, then $K(M, n-1)\neq \l 1\r$.\end{thm}

Flechsig proved
that for $M\in {\cal C}_0$, if $M$ has a cone point
then $K(M, n-1)\neq \l 1\r$.

To prove Theorem \ref{ig},  first 
we use the presentation of surface pure braid groups from \cite{Bel04},
to produce an explicit set of generators and relations
of the surface orbifold pure braid group of  $M\in {\cal C}_1$. Then, 
we apply the Fadell-Neuwirth fibration theorem for surfaces to
deduce the exact sequence. To prove $K(M, n-1)\neq \l 1\r$, we show
that the homomorphism $\pi_1^{orb}(\wt D)\to \pi_1^{orb}(\wt M)$, for
$D\subset M$ a disc with a cone point, is injective, and then use the above
result of Flechsig for $D$.

In \cite{Rou20} we used a 
stretching technique, and   
there we did not need to appeal to the Fadell-Neuwirth fibration theorem for
punctured complex plane.
However, the stretching technique is not applicable for 
genus $\geq 1$ orbifolds, as in this case the movement of the strings of a braid is more
complicated and a pictorial description is not possible. 

\begin{rem}\label{les}{\rm If 
    $M$ is an aspherical $2$-manifold then using the exact sequence (\ref{1.1}) and by an induction
    on $n$, it follows that
    $\pi_k(PB_n(M))=\l 1\r$, for all $k\geq 2$, and hence (\ref{1.1}) becomes a short exact sequence of fundamental
    groups. Now let $M\in {\cal C}_0$ with a cone point.
    Then by \cite{JF23},  $K(M, n-1)\neq \l 1\r$. We have given
    examples of $M\in {\cal C}_0$ with cone points in
    [\cite{Rou23-1}, Theorem 1.2],
    such that 
    $\pi^{orb}_k(PB_n(M))=\l 1\r$, for all $k\geq 2$. Therefore, for such $M$ there is no long
    exact sequence induced by $f(M)$, similar to (\ref{1.1}), of orbifold homotopy groups.}\end{rem}

  \section{Proof}\label{proofs}
In this section we give the proof of Theorem \ref{ig}.

\begin{proof}[Proof of Theorem \ref{ig}] We have assumed that the underlying
  space of the orbifold is orientable. 
  Also, we assume that the orbifold has a 
  cone point, since otherwise the Fadell-Neuwirth fibration theorem will be applicable.

  First we give the proof in the case when $\pi_1^{orb}(M)$ is
  finitely generated.
  
  \medskip
  \noindent
  {\bf Case 1: $\pi_1^{orb}(M)$ is finitely generated:}  
Under this hypothesis, we can assume that $M$ is of the type $M_r^s$, where
$M_r^s$ is a $2$-dimensional orientable orbifold of genus $g\geq 1$, with
$r$ number of punctures and $s\geq 1$ number of cone points of orders
$q_1,q_2,\ldots,q_s$.

The plan of the proof is as follows.
We start with the direct way to define the orbifold fundamental group of
an orbifold from \cite{Thu91}. Then, we  give an explicit
presentation of the orbifold fundamental group of $PB_n(M)$, using the
presentation of the surface pure braid group of a surface, from \cite{Bel04}. The
homomorphism $\pi_1^{orb}(PB_n(M))\to \pi_1^{orb}(PB_{n-1}(M))$ is
then described using this presentation. Finally, we apply the Fadell-Neuwirth
fibration theorem for surfaces and some group theoretic arguments,
to conclude the proof in this
case.

Let $P=\{1,2,\ldots, n\}$ and $\{x_1,x_2,\ldots, x_n\}\subset M_k^0$ be a fixed subset of
$n$ distinct points. Consider the set $S$ of all
continuous maps $\gamma:P\times I\to M_k^0$,
satisfying the following conditions.
\begin{itemize}
  \item $\gamma(i,t)\neq \gamma(j,t)$ for all $t\in I$ and for $i\neq j$.

\item $\gamma(i,0)=\gamma(i,1)=x_i$ for all $i\in P$.
\end{itemize}
\noindent
Note that $\gamma$ is a loop in $PB_n(M_k^0)$ 
based at $(x_1,x_2,\ldots, x_n)\in PB_n(M_k^0)$. We call it a {\it braid loop}.

Given two braid loops $\gamma_0$ and $\gamma_1$, a homotopy between them,   
  fixing end points, is a map $F:P\times I\times I\to M_k^0$ satisfying the following conditions.
  \begin{itemize}
  \item For each $t\in I$, $F|_{P\times I\times \{t\}}\in S$.
  \item $F|_{P\times I\times \{0\}}=\gamma_0$ and  $F|_{P\times I\times \{1\}}=\gamma_1$.
  \end{itemize}
  Composition between the braid loops $\gamma_0$ and $\gamma_1$ is
  defined as $\gamma=\gamma_0*\gamma_1$, where for $i\in P$, the following are
  satisfied.
  \begin{itemize}
        \item $\gamma(i,t)=\gamma_0(i, 2t)$ for $0\leq t \leq \frac{1}{2}$.
      \item $\gamma(i, t)=\gamma_1(i, 2t-1)$ for $\frac{1}{2}\leq t \leq 1$.
      \end{itemize}
      Clearly, the homotopy classes under this composition
      law of braid loops, give $\pi_1(PB_n(M_k^0))$, with base point $(x_1,x_2,\ldots, x_n)$.
      We are not including the base point in the notation of the fundamental group, as 
      it will remain fixed during the proof. For $PB_{n-1}(M_k^0)$ the
      base point will be $(x_1,x_2,\ldots, x_{n-1})$.

 A presentation of the group
  $\pi_1(PB_n(M_k^0))$, in terms of generators and relations is given in
  [\cite{Bel04}, Theorem 5.1]. In Figure 1 we show all the generators of
  this group. Let the list of relations be $\bf R$. We do not
  reproduce the list of relations here, as
we will not need its explicit descriptions. That is, we have the following.

$$\pi_1(PB_n(M_k^0))=\l A^j_i, B^j_i,C_l^m,P_i^p; i=1,2,\ldots, n;$$
$$j=1,2,\ldots, g; m < l, m=1,2,\ldots, n-1; p=1,2,\ldots, k\ |\ {\bf R}\r.$$

Now let $k=r+s$ and $s\geq 1$.
Next, we replace the punctures from $r+1$ to $r+s$ by cone points of
orders $q_1,q_2,\ldots, q_s$,
respectively, as shown in Figure 2. We also rename the generators with a `bar' to avoid
confusion.

     \medskip
\centerline{\includegraphics[height=5cm,width=10cm,keepaspectratio]{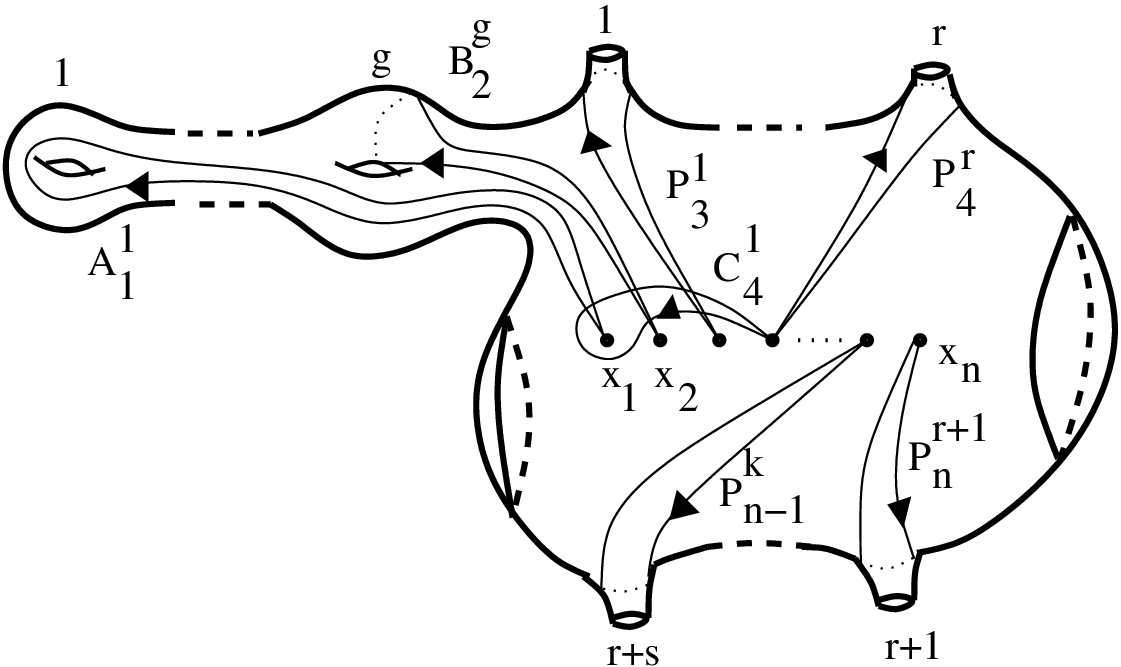}}

\centerline{Figure 1: Generators of the surface pure braid group of $M_k^0$.}

\medskip

Then, a presentation of the surface pure orbifold braid group
$\pi_1^{orb}(PB_n(M_r^s))$ is obtained from the presentation
of $\pi_1(PB_n(M_{r+s}^0))$ by adding the extra relations  
$(\bar P_i^{r+j})^{q_j}$, for $i=1,2,\ldots, n; j=1,2,\ldots s$. Since, by
definition of orbifold fundamental group (see \cite{All02} or \cite{Thu91}), if a loop
$\eta$ circles around a cone point of order $q$, then $\eta^q=1$ appears as a
relation. The only
difference between a puncture and a cone point on a $2$-dimensional orbifold,
which reflects on the orbifold fundamental group, is this finite order
relation. Hence, we have the
following. 

$$\pi_1^{orb}(PB_n(M_r^s))=$$
$$\l \bar A^j_i, \bar B^j_i, \bar C_l^m, \bar P_i^p; i=1,2,\ldots, n; j=1,2,\ldots, g;m < l, m=1,2,\ldots, n-1; $$
$$p=1,2,\ldots, r+s\ |\ \bar {{\bf R}}\cup \{(\bar P_i^{r+j})^{q_j},\
\text{for}\ i=1,2,\ldots, n, j=1,2,\ldots s\}\r.$$
\noindent
Here $\bar {{\bf R}}$ is the same set of relations as in $\bf R$ but the generators
are replaced with a `bar'.

\medskip

\centerline{\includegraphics[height=5cm,width=10cm,keepaspectratio]{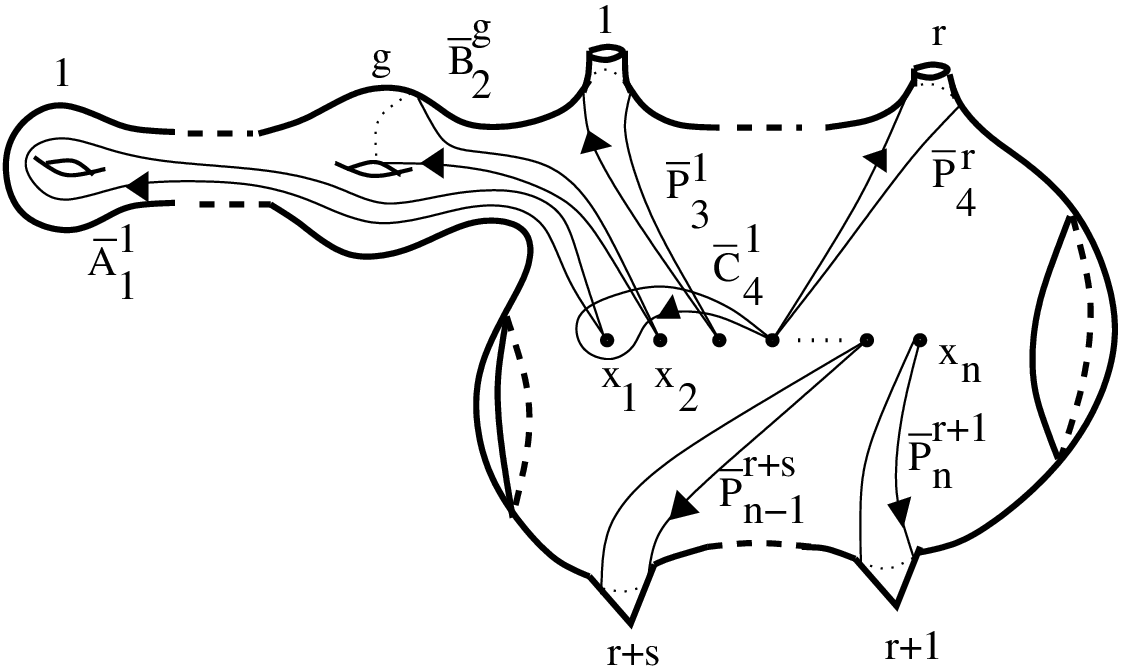}}

\centerline{Figure 2: Generators of the surface pure orbifold braid group.}

\medskip

Clearly, now there is the following surjective homomorphism, which sends a generator $X$ of
$\pi_1(PB_n(M_{r+s}^0))$ to $\bar X$, 
$$g_n:\pi_1(PB_n(M_{r+s}^0))\to \pi_1^{orb}(PB_n(M_r^s)).$$
This homomorphism is also obtained by applying the $\pi_1^{orb}$ functor
on the inclusion $PB_n(M_{r+s}^0)\to PB_n(M_r^s)$.

Therefore, we have the following commutative diagram.

\centerline{
  \xymatrix{\pi_1(PB_n(M_{r+s}^0))\ar[r]^{\!\!\!\!\!\! f_n}\ar[d]^{g_n}&\pi_1(PB_{n-1}(M_{r+s}^0))\ar[d]^{g_{n-1}}\\
    \pi_1^{orb}(PB_n(M_r^s))\ar[r]^{f^o_n}&\pi_1^{orb}(PB_{n-1}(M_r^s)).}}
\noindent
Here, $f_n=f(M_{r+s}^0)_*$, which is induced by the projection to the first $n-1$ coordinates. 
Recall that, $f(M_{r+s}^ 0)$ is a fibration by the Fadell-Neuwirth fibration theorem.

$f_n^o$ is also induced by a similar projection. However, we have
seen in [\cite{Rou20}, Proposition 2.11]) that the corresponding
homomorphism of this projection, on an associated   
configuration Lie groupoid is not a fibration.

Since $f_n$ and $g_{n-1}$ are both surjective, together with the Fadell-Neuwirth
fibration theorem we get the following commutative diagram.

\centerline{
  \xymatrix{\pi_1(\wt M_{r+s}^0)\ar[r]&\pi_1(PB_n(M_{r+s}^0))\ar[r]^{\!\!\!\!\! f_n}\ar[d]^{g_n}&\pi_1(PB_{n-1}(M_{r+s}^0))\ar[d]^{g_{n-1}}\ar[r]&1\\
    &\pi_1^{orb}(PB_n(M_r^s))\ar[r]^{f^o_n}&\pi_1^{orb}(PB_{n-1}(M_r^s))\ar[r]&1.}}


Here $\wt M_{r+s}^0=M_{r+s+n-1}^0$ is the fiber over the point
$(x_1,x_2,\ldots, x_{n-1}$) of the projection map. The points $x_1,x_2,\ldots, x_{n-1}$
are replaced with punctures in $\wt M_{r+s}^0$.

It is clear that, the kernel of $f_n$ is normally generated by the
following generators of $\pi_1(PB_n(M_{r+s}^0))$. See Figure 3.
$$\{A^j_n, B^j_n,C_n^m,P_n^p; j=1,2,\ldots, g; m=1,2,\ldots, n-1; p=1,2,\ldots, r+s\}.$$


\centerline{\includegraphics[height=5cm,width=10cm,keepaspectratio]{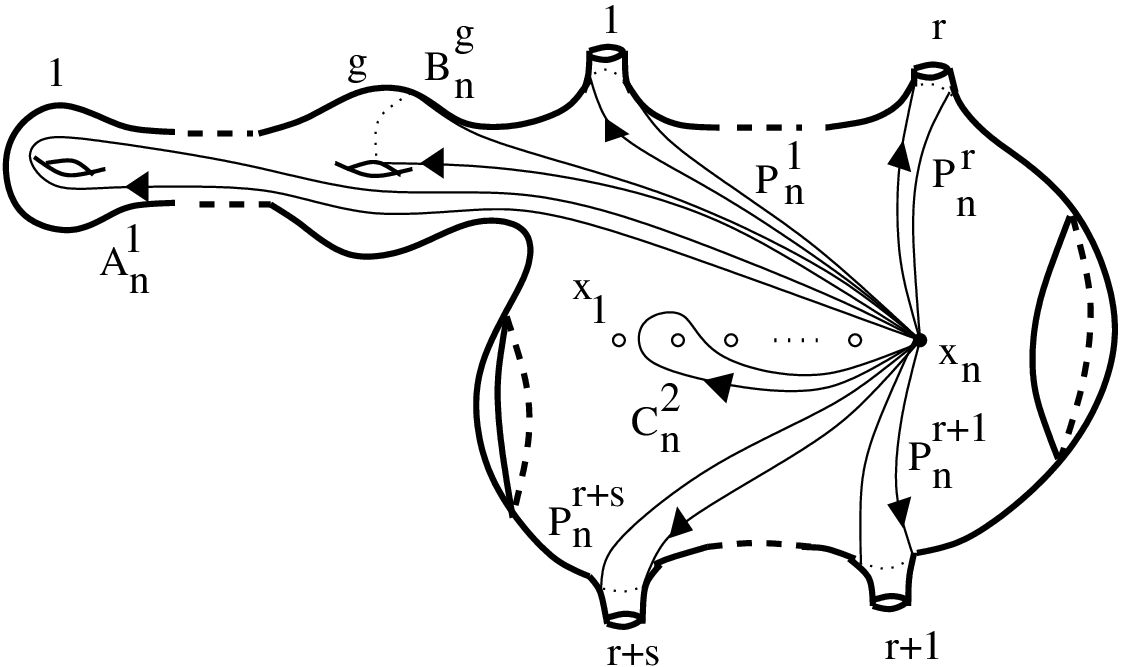}}

\centerline{Figure 3: Generators of the kernel of $f_n$.}

Since $f_n$ is induced by
a fibration, $\pi_2(PB_k(M_{r+s}^0))=0$ for all $k$ (Remark \ref{les}).

Therefore, the above generators, which generate $\pi_1(\wt M_{r+s}^0)$, generate a normal subgroup of
$\pi_1(PB_n(M_{r+s}^0))$.

On the other hand, the kernel of $f^o_n$ is normally generated by the following generators
of $\pi_1^{orb}(PB_n(M_r^s))$. See Figure 4.
$$\{\bar A^j_n, \bar B^j_n, \bar C_n^m, \bar P_n^p;
j=1,2,\ldots, g; m=1,2,\ldots, n-1; p=1,2,\ldots, r+s\}.$$

Next, note that the image of $\pi_1^{orb}(\wt M_r^s)$ in $\pi_1^{orb}(PB_n(M_r^s))$ is obtained
from the above presentation of $\pi_1(\wt M_{r+s}^0)$, by adding the finite order
relations $(\bar P_n^{r+j})^{q_j}$, for $j=1,2,\ldots s$. But this is
the presentation of $\pi_1^{orb}(\wt M_r^s)$, where $\wt M_r^s=M_{r+n-1}^s$.
Here also, $x_1,x_2,\ldots, x_{n-1}$
are replaced with punctures in $\wt M_r^s$.

\medskip

\centerline{\includegraphics[height=5cm,width=10cm,keepaspectratio]{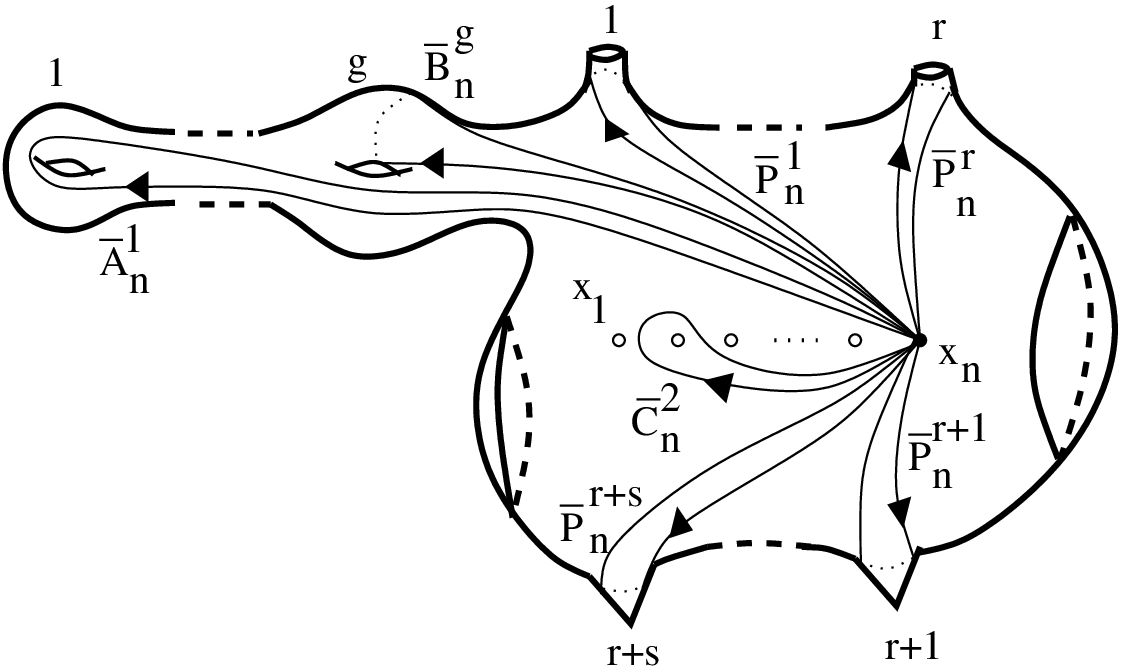}}

\centerline{Figure 4: Generators of the kernel of $f_n^o$.}

\medskip

Since $g_n$ is a surjective homomorphism, and
it sends the image of $\pi_1(\wt M_{r+s}^0)$ onto the image of $\pi_1^{orb}(\wt M_r^s)$, we
see that the above generators generate a normal subgroup of $\pi_1^{orb}(PB_n(M_r^s))$.
Hence, the kernel of $f^o_n$ is exactly the image
of $\pi_1^{orb}(\wt M_r^s)$. 

Therefore, from the above argument we get the following commutative diagram.

\centerline{
  \xymatrix@C-1pc{1\ar[r]&1\ar[r]\ar[d]&\pi_1(\wt M_{r+s}^0)\ar[r]\ar[d]&\pi_1(PB_n(M_{r+s}^0))\ar[r]^{\!\!\!\!\! f_n}\ar[d]^{g_n}&\pi_1(PB_{n-1}(M_{r+s}^0))\ar[d]^{g_{n-1}}\ar[r]&1\\
    1\ar[r]&K(M_r^s, n-1)\ar[r]&\pi_1^{orb}(\wt M_r^s)\ar[r]&\pi_1^{orb}(PB_n(M_r^s))\ar[r]^{f^o_n}&\pi_1^{orb}(PB_{n-1}(M_r^s))\ar[r]&1.}}

\medskip

Here, $K(M_r^s, n-1)$ is denoting the kernel
of $\pi_1^{orb}(\wt M_r^s)\to \pi_1^{orb}(PB_n(M_r^s))$.
This completes the proof of Theorem \ref{ig} in the finitely generated case.

Next, we give the proof in the case 
when $\pi_1^{orb}(M)$ is infinitely generated. We will use the
finitely generated case in a direct limit argument.

\medskip
\noindent
{\bf Case 2: $\pi_1^{orb}(M)$ is infinitely generated:}
We write $M$ as an increasing union $\bigcup_{i\in {\Bbb N}}M_{r_i}^{s_i}$ of suborbifolds $M_{r_i}^{s_i}$.
Here $r_1\leq r_2\leq\cdots$, 
$s_1\leq s_2\leq\cdots$ and $M_{r_i}^{s_i}$ has finite ($\geq 1$) genus. Clearly, 
$\pi_1^{orb}(M_{r_i}^{s_i})$ is infinite and finitely generated.

Now, we have inclusions $PB_n(M_{r_i}^{s_i})\subset PB_n(M_{r_j}^{s_j})$ for $i \leq j$,
and $PB_n(M)=\bigcup_{i\in {\Bbb N}} PB_n(M_{r_i}^{s_i})$.
Then, there is the following commutative diagram.

\centerline{
  \xymatrix@C+1pc{PB_n(M_{r_i}^{s_i})\ar[r]^{\!\!\! f(M_{r_i}^{s_i})}\ar[d]&PB_{n-1}(M_{r_i}^{s_i})\ar[d]\\
PB_n(M_{r_j}^{s_j})\ar[r]^{\!\!\!\!\!\!f(M_{r_j}^{s_j})}&PB_{n-1}(M_{r_j}^{s_j}).}}

Using Case $1$, the above diagram induces the following commutative
diagram. Next, we take the direct limit of this directed system of four-term exact sequences, and
complete the proof of the four-term exact sequence of 
Theorem \ref{ig}.

\centerline{
  \xymatrix@C-1.1pc{1\ar[r]&K(M_{r_i}^{s_i}, n-1)\ar[r]\ar[d]&\pi_1^{orb}(\wt M_{r_i}^{s_i})\ar[r]\ar[d]&\pi_1^{orb}(PB_n(M_{r_i}^{s_i}))\ar[r]^{\!\!\!\!\! f^o_n}\ar[d]&\pi_1^{orb}(PB_{n-1}(M_{r_i}^{s_i}))\ar[r]\ar[d]&1\\
1\ar[r]&K(M_{r_j}^{s_j}, n-1)\ar[r]&\pi_1^{orb}(\wt M_{r_j}^{s_j})\ar[r]&\pi_1^{orb}(PB_n(M_{r_j}^{s_j}))\ar[r]^{\!\!\!\!\! f^o_n}&\pi_1^{orb}(PB_{n-1}(M_{r_j}^{s_j}))\ar[r]&1.}}

We now proceed to prove that
$K(M, n-1)\neq \l 1\r$, if $M\in {\cal C}_1$ has a cone point.

Let $D$ be a small open disc in $M$ containing one cone point, and assume that
$\partial D$ contains no cone point. 
Then, we have the following commutative diagram, induced by the inclusion $D\subset M$.

\centerline{
  \xymatrix@C-.7pc{1\ar[r]&K(D, n-1)\ar[r]\ar[d]&\pi_1^{orb}(\wt D)\ar[r]\ar[d]&\pi_1^{orb}(PB_n(D))\ar[r]\ar[d]&\pi_1^{orb}(PB_{n-1}(D))\ar[r]\ar[d]&1\\
    1\ar[r]&K(M, n-1)\ar[r]&\pi_1^{orb}(\wt M)\ar[r]&\pi_1^{orb}(PB_n(M))\ar[r]&\pi_1^{orb}(PB_{n-1}(M))\ar[r]&1.}}

We now need the following lemma.

\begin{lemma}\label{injective}The homomorphism 
$\pi_1^{orb}(\wt D)\to \pi_1^{orb} (\wt M)$
is injective.\end{lemma}

Before we give the proof of the lemma, note that,
since Flechsig proved in \cite{JF23} that $K(D, n-1)$ is non-trivial,
we conclude that $K(M, n-1)$
is also non-trivial.

This completes the proof of the theorem.\end{proof}

\begin{proof}[Proof of Lemma \ref{injective}] We use Van-Kampen theorem for orbifolds. Let $\wt M_1=\wt M-\wt D$ and
$\wt D_1=\wt D\cup \partial D$. We now
check that $\partial \wt D_1\subset X$ induces an injective homomorphism in
orbifold fundamental groups, for $X=\wt D_1$ or $\wt M_1$. For 
$X\in {\cal C}_1$, let $\hat X$ be the smooth $2$-manifold after
disregarding the cone structure of the cone points in $X$, that is, $\hat X$ is
the underlying space of $X$. Then $\pi_1(\hat X)$
is obtained from $\pi_1^{orb}(X)$ by adding the relations $\alpha^q=1$ for any
small loop $\alpha$ in $X$ around a cone point of order $q$. It is well known
that $\partial \wt D_1\subset \hat X$ induces an injective homomorphism in
fundamental groups, for $X=\wt D_1$ or $\wt M_1$. Therefore, we have the 
following commutative diagram, where the horizontal homomorphism is injective. Hence
the slanted homomorphism is also injective.

\centerline{
  \xymatrix{&\pi_1^{orb}(X)\ar[d]\\
    \pi_1(\partial \wt D_1)\ar[ur]\ar[r]&\pi_1(\hat X).}}

Note that $\wt M=\wt D_1\cup \wt M_1$, and hence applying Van-Kampen theorem
we get that $\pi_1^{orb}(\wt M)$ is isomorphic to a generalized free product
$G_1*_HG_2$ where, $G_1\simeq \pi_1^{orb}(\wt D_1)\simeq \pi_1^{orb}(\wt D)$,
$H\simeq \pi_1(\partial \wt D_1)$ and $G_2\simeq \pi_1^{orb}(\wt M_1)$.
Since the homomorphisms $H\to G_i$, for $i=1,2$ are both injective, it is
known that $G_1\to G_1*_HG_2$ is also injective.
This completes the proof of the lemma.
\end{proof}

\begin{rem}\label{k}{\rm Recall that for a manifold $M$ of dimension $\geq 2$,
    the Fadell-Neuwirth fibration theorem also says
    that the projection to the first $k$-coordinates, for $k < n$, gives a
    fibration $PB_n(M)\to PB_k(M)$, with fibers homeomorphic  to
    $PB_{n-k}(M-\{k\ \text{points}\})$. Our proof can be suitably modified to give a
    proof of the following exact sequence, where $M\in {\cal C}_1$
    and $\wt M=M-\{k\ \text{smooth points}\}$.

    \centerline{
 \xymatrix@C-.7pc{1\ar[r]&K(M,k)\ar[r]&\pi_1^{orb}(PB_{n-k}(\wt M))\ar[r]&\pi_1^{orb}(PB_n(M))\ar[r]&\pi_1^{orb}(PB_k(M))\ar[r]&1.}}}\end{rem}

  \newpage
\bibliographystyle{plain}
\ifx\undefined\bysame
\newcommand{\bysame}{\leavevmode\hbox to3em{\hrulefill},}
\fi

\end{document}